\documentclass[11pt]{article}
\usepackage{amsmath}
\usepackage{amssymb}
\usepackage{amscd}
\usepackage{latexsym}
\usepackage{theorem}
\usepackage{setspace}
\usepackage{epsfig}
\usepackage{psfrag}
\usepackage{tikz}
\usepackage{multicol}
\usepackage{todonotes}
\usepackage{euscript}           

\tikzstyle{invisivertex} = [shape=rectangle, minimum size=0pt, inner sep=2pt]
\tikzstyle{label} = [violet!60!white, shape=rectangle, minimum size=0pt, inner sep=2pt]
\tikzstyle{point}=[draw, black, fill,shape=circle, minimum size=4pt, inner sep=0pt]
\tikzstyle{puncture}=[draw, black, fill=white,shape=circle, minimum size=4pt, inner sep=0pt]
\newcommand{\drawstar}[3]{%
    \begin{scope}[shift={#1}]
        \pgfmathsetmacro{\outerradius}{#2}
        \pgfmathsetmacro{\starangle}{360/7}  
        \pgfmathsetmacro{\innerradius}{\outerradius * sin(18) / sin(108)} 

        \path[#3] (90:\outerradius)
        \foreach \i in {1,...,10} {
            -- ({90+\i*\starangle - 180/\starangle}:\innerradius)
            -- ({90+\i*\starangle}:\outerradius)
        } -- cycle;
    \end{scope}
}

\newtheorem{theorem}{Theorem}[section]
\newtheorem{lemma}[theorem]{Lemma}
\newtheorem{proposition}[theorem]{Proposition}
\newtheorem{corollary}[theorem]{Corollary}

{\theorembodyfont{\rmfamily}
\theoremstyle{plain}

\newtheorem{example}[theorem]{Example}

\newtheorem{remark}[theorem]{Remark}

}

\newenvironment{proof-sal-pair}{\noindent {\bf Proof of Theorem \ref{thm:sal-pair}:}}{\qed \par}
\newenvironment{proof-z-pair}{\noindent {\bf Proof of Theorem \ref{thm:z-pair}:}}{\qed \par}

\newcommand{\C}{{\mathbb{C}}}

\newcommand{\R}{{\mathbb{R}}}

\newcommand{\cT}{\mathcal{T}}

\newcommand{\Sep}{\operatorname{Sep}}

\newcommand{\cat}{\EuScript}    
\newcommand{\Top}{\cat{T}}

\newcommand{\cA}{\mathcal{A}}

\newcommand{\hocolim}{\operatorname{hocolim}}

\newcommand{\ra}{\rightarrow}                   

\newcommand{\bdot}{\bullet}
\DeclareMathOperator{\srep}{srep}
\DeclareMathOperator{\Op}{Op}

\newcommand{\qed}{\hfill \mbox{$\Box$}\medskip\newline}
\newenvironment{proof}{\noindent {\bf Proof:}}{\qed \par}

\renewcommand{\and}{\qquad\text{and}\qquad}

\newcommand{\cK}{\mathcal{K}}
\newcommand{\cZ}{\mathcal{Z}}
\newcommand{\cM}{\mathcal{M}}

\newcommand{\Sal}{\operatorname{Sal}}
\newcommand{\cL}{\mathcal{L}}
\newcommand{\excise}[1]{}

\begin{document}
\spacing{1.2}
\noindent{\Large\bf Salvetti complexes for conditional oriented matroids}\\

\noindent{\bf Galen Dorpalen-Barry}\footnote{Supported by NSF grant DMS-2039316.}\\
Department of Mathematics, Texas A\&M, College Station, TX 77840\\
email: dorpalen-barry@tamu.edu
\vspace{.1in}

\noindent{\bf Dan Dugger}\footnote{Supported by NSF grant DMS-2039316.}\\
Department of Mathematics, University of Oregon, Eugene, OR 97403\\
email: ddugger@uoregon.edu
\vspace{.1in}

\noindent{\bf Nicholas Proudfoot}\footnote{Supported by NSF grants DMS-2039316, DMS-2053243, and DMS-2344861.}\\
Department of Mathematics, University of Oregon, Eugene, OR 97403\\
email: njp@uoregon.edu\\

{\small
\begin{quote}
\noindent {\em Abstract:}
We give a new proof of the fact that the complement of the complexification of a real hyperplane
arrangement is homotopy equivalent to the Salvetti complex of the associated oriented matroid.
Our proof involves no choices, is relatively easy to visualize, and generalizes to the setting of conditional oriented matroids.
\end{quote} }

{\small
\begin{quote}
\end{quote} }

\section{Introduction}
Let $\cA$ be a finite collection of affine hyperplanes in a real vector space $V$.
For each $H\in\cA$, let $H_0$ be the linear hyperplane obtained by translating $H$ to the origin, and let 
$$H\otimes\C := \{v+iw\in V\otimes \C\mid \text{$v\in H$ and $w\in H_0$}\}.$$
If $H$ is cut out of $V$ by an affine linear function $f:V\to\R$, then $H\otimes\C$ is cut out of $V\otimes\C$ by an affine linear function
with the same coefficients.
We consider three topological spaces associated with $\cA$:
\begin{itemize}
\item The complexified complement $\cM(\cA) := V\otimes\C \setminus \bigcup_{H\in\cA} H \otimes \C$.
\item The Salvetti complex $|\Sal(\cA)|$, the order complex of a poset $\Sal(\cA)$ defined in terms of the pointed oriented matroid associated with $\cA$.
\item The non-Hausdorff space $\cZ(\cA)$ obtained by gluing together finitely many copies of $V$, one for each chamber of $\cA$,
with $V_C$ and $V_{C'}$ glued along the complement of the hyperplanes that separate $C$ from $C'$.
\end{itemize}
These three spaces are related as follows.

\begin{theorem}\label{thm:sal}{\em \cite{Salvetti}} The simplicial complex $|\Sal(\cA)|$ is homotopy equivalent to $\cM(\cA)$.
\end{theorem}

\begin{theorem}\label{thm:z}{\em \cite{non-Haus}}
The manifold $\cM(\cA)$ is the total space of a principal bundle over the space $\cZ(\cA)$ with structure group $V^*$, the linear 
dual of $V$.  In particular, $\cM(\cA)$ is weakly homotopy equivalent to $\cZ(\cA)$.
\end{theorem}


The original proof of Theorem \ref{thm:sal} proceeds by constructing an inclusion of $|\Sal(\cA)|$ into $\cM(\cA)$ and showing that the image
is a deformation retract of the target.  While extremely concrete, this proof relies on a number of arbitrary choices.
In the case where all of the hyperplanes pass through the origin, 
Paris \cite{Paris} provided a new proof of Theorem \ref{thm:sal} by constructing a good open cover of $\cM(\cA)$ and identifying
$|\Sal(\cA)|$ with the nerve of this cover.

In this note, we will provide yet another proof of Theorem \ref{thm:sal}, also based on a variant of the nerve lemma.  Rather than working with 
Paris's cover of $\cM(\cA)$, we will work with the defining cover of $\cZ(\cA)$.  This is not a good cover in the technical sense, because the intersections
of the open sets are not connected.  However, it is completely canonical, and the topology of the intersections
interacts well with the combinatorics of oriented matroids.
Our analysis of this cover will allow us to conclude that $|\Sal(\cA)|$ is weakly homotopy equivalent
to $\cZ(\cA)$, and therefore also homotopy equivalent to $\cM(\cA)$ by Theorem \ref{thm:z}.  A sketch of a related approach appears in
\cite[Remark 3.4.3]{bug}.

One nice feature of this perspective is that it generalizes easily to a broader setting, as we now describe.
Fix a nonempty convex open subset $\cK \subset V$, and let $$\cM(\cA,\cK) := \{v+iw\in M(\cA)\mid v\in\cK\}.$$
Note that $\cM(\cA,V) = \cM(\cA)$.
The definition of the Salvetti poset generalizes to give us a new poset $\Sal(\cA,\cK)$, now defined in terms of the associated
{\bf conditional oriented matroid} \cite{BCK}.
We then define $\cZ(\cA,\cK)$ in a manner analogous to the definition of $\cZ(\cA)$,
by gluing together copies of $\cK$ rather than of $V$.  With little additional work, we prove 
the following generalizations of Theorems \ref{thm:sal} and \ref{thm:z}.

\begin{theorem}\label{thm:sal-pair} The simplicial complex $|\Sal(\cA,\cK)|$ is homotopy equivalent to the manifold $\cM(\cA,\cK)$.
\end{theorem}

\begin{theorem}\label{thm:z-pair}
The manifold $\cM(\cA,\cK)$ is canonically isomorphic to a principal bundle over the space $\cZ(\cA,\cK)$ with structure group $V^*$,
and is therefore weakly homotopy equivalent to $\cZ(\cA,\cK)$.
\end{theorem}

\begin{remark}
Even if we are only interested in the case where $\cK = V$, the proof that we will give of Theorem \ref{thm:sal-pair} will force us to think about
the pair $(\cA,\cK_z)$, where $\cK_z$ is a certain cone in $V$ that depends on the choice of a point $z\in\cZ(\cA,V)$.  Hence the generalization
of Salvetti's theorem to pairs $(\cA,\cK)$ and their associated conditional oriented matroids feels natural.
\end{remark}


\section{Conditional oriented matroids}
Let $E$ be a finite set.  An element of $\{-,0,+\}^E$ is called a {\bf sign vector}.  Given a pair of sign vectors $X$ and $Y$,
the {\bf composition} $X\circ Y$ is a new sign vector with $e^{\text{th}}$ entry
$$(X\circ Y)_e := \begin{cases}
X_e & \text{if $X_e\neq 0$}\\
Y_e & \text{otherwise}.
\end{cases}$$
Given two sign vectors $X$ and $Y$, their {\bf separating set} is
$$\Sep(X,Y) := \{e\in E\mid \{X_e,Y_e\}=\{-,+\}\}.$$
A nonempty collection $\cL$ of sign vectors is called a {\bf conditional oriented matroid} \cite{BCK} if it satisfies the following two conditions:
\begin{itemize}
\item If $X,Y\in\cL$, then $X\circ (-Y)\in\cL$.
\item If $X,Y\in\cL$ and $e\in\Sep(X,Y)$, then there exists $Z\in \cL$ such that $Z_e=0$ and $Z_f = (X\circ Y)_f$ for all $f\in E\setminus\Sep(X,Y)$.
\end{itemize}
Elements of $\cL$ are called {\bf covectors}.
If in addition $\cL$ contains the zero vector, then $\cL$ is called an {\bf oriented matroid}.

\begin{example}\label{realizable}
Let $V$ be a real vector space and $\cA = \{H_e\mid e\in E\}$ a finite multiset of cooriented affine hyperplanes in $V$, indexed by $E$.
More precisely, for each $e\in E$, we have an affine hyperplane $H_e$, a positive side $H_e^+$, and a negative side $H_e^-$, with 
$V = H_e^+\sqcup H_e\sqcup H_e^-$.  For each sign vector $X$, we define
$$F_X := \bigcap_{X_e = -} H_e^- \cap \bigcap_{X_e = 0} H_e \cap \bigcap_{X_e = +} H_e^+$$
to be the corresponding face of $\cA$.  Given a nonempty convex open subset $\cK\subset V$, we have a conditional oriented matroid
$$\cL(\cA,\cK) := \{X\mid F_X \cap \cK \neq \emptyset\}.$$
This is an oriented matroid if and only if there exists a point in $\cK$ that is contained in all of the hyperplanes,
in which case $\cL(\cA,\cK) = \cL(\cA, V)$.
\end{example}

There is a natural partial order on sign vectors defined by putting $X\leq Y$ if and only if $X\circ Y = Y$, and this restricts to a partial order
on the covectors of any conditional oriented matroid.  In the special case of Example \ref{realizable}, we have $X\leq Y\in\cL(\cA,\cK)$ 
if and only if $F_X$ is contained in the closure of $F_Y$.

Given a poset $P$, the {\bf order complex} $|P|$ is the simplicial complex whose $r$-simplicies consist of chains $x_0<x_1<\cdots<x_r$ in $P$.
The following result is proved in \cite[Proposition 15]{BCK}.

\begin{proposition}\label{face poset}
For any conditional oriented matroid $\cL$, the order complex $|\cL|$ is contractible.
\end{proposition}

\begin{remark}
The proof in \cite{BCK} of Proposition \ref{face poset} proceeds by reducing to the case where $\cL$ is semisimple, which means
that the coordinate functions $\chi_e:\cL\to\{-,0,+\}$ are nonconstant and distinct.  This is sufficient because every conditional oriented matroid $\cL$
has a canonical semisimplification, and the posets of covectors of the semisimplification is isomorphic to that of $\cL$.
\end{remark}

\begin{example}\label{two points 1}
Suppose that $\cA$ consists of two distinct points on a line $V$, each oriented so that the positive side is on the right.
We have five covectors:
$$A = (--),\quad H = (0-),\quad B = (+-),\quad H' = (+0), \quad C = (++).$$
In the picture below, we label each face $F_X$ with the corresponding covector $X$.

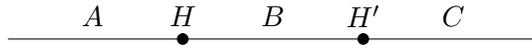
\begin{center}
\begin{tikzpicture}[scale=.6]

\node[invisivertex] (R-) at (-6,0){};
\node[invisivertex] (R+) at (6,0){};

\node[invisivertex ] (H) at (-2,.5){$H$};
\node[invisivertex ] (H') at (2,.5){$H'$};

\node[invisivertex] (H) at (-4,.5){$A$};
\node[invisivertex] (H') at (0,.5){$B$};
\node[invisivertex] (H') at (4,.5){$C$};

\node[point] (H) at (-2,0){};
\node[point] (H') at (2,0){};

\path[-] (R-) edge [] node[above] {} (R+);
\end{tikzpicture}

\excise{
\begin{tikzpicture}[scale=.6]

\node[invisivertex] (R-) at (-6,0){};
\node[invisivertex] (R+) at (6,0){};

\node[invisivertex ] (H) at (-2,.5){$0+$};
\node[invisivertex ] (H') at (2,.5){$-0$};

\node[invisivertex] (H) at (-4,.5){$++$};
\node[invisivertex] (H') at (0,.5){$-+$};
\node[invisivertex] (H') at (4,.5){$--$};

\node[point] (H) at (-2,0){};
\node[point] (H') at (2,0){};

\path[-] (R-) edge [] node[above] {} (R+);
\end{tikzpicture}
}
\end{center}
The Hasse diagram of the poset $\cL(\cA,V)$ looks like Cassiopaea:

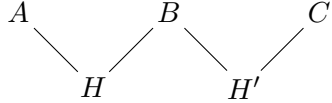
\begin{center}
\begin{tikzpicture}[]

\node[invisivertex ] (H) at (-1,0) {$H$};
\node[invisivertex ] (H') at (1,0){$H'$};

\node[invisivertex] (A) at (-2,1){$A$};
\node[invisivertex] (B) at (0,1){$B$};
\node[invisivertex] (C) at (2,1){$C$};

\path[-] (H) edge [] node[above] {} (A);
\path[-] (H) edge [] node[above] {} (B);
\path[-] (H') edge [] node[above] {} (B);
\path[-] (H') edge [] node[above] {} (C);
\end{tikzpicture}
\excise{\qquad
\begin{tikzpicture}[]
\drawstar{(-1,0)}{.5}{draw=yellow, fill=yellow};
\drawstar{(1,0)}{.5}{draw=yellow, fill=yellow};
\drawstar{(-2,1)}{.5}{draw=yellow, fill=yellow};
\drawstar{(0,1)}{.5}{draw=yellow, fill=yellow};
\drawstar{(2,1)}{.5}{draw=yellow, fill=yellow};
\node[invisivertex ] (H) at (-1,0){$0+$};
\node[invisivertex ] (H') at (1,0){$-0$};

\node[invisivertex] (A) at (-2,1){$++$};
\node[invisivertex] (B) at (0,1){$-+$};
\node[invisivertex] (C) at (2,1){$--$};

\path[-] (H) edge [] node[above] {} (A);
\path[-] (H) edge [] node[above] {} (B);
\path[-] (H') edge [] node[above] {} (B);
\path[-] (H') edge [] node[above] {} (C);
\end{tikzpicture}}
\end{center}
Since the longest chains consist of two elements, the order complex looks exactly like the Hasse diagram, and one can see that it is contractible.
\end{example}

Let $\cL$ be a conditional oriented matroid.  A covector that is maximal with respect to our partial order is called a {\bf tope}.
In the case where $\cL = \cL(\cA,\cK)$, the topes are in bijection with the connected components of $\cK \setminus\bigcup_{e\in E}H_e$ via the correspondence
$T\mapsto F_T\cap\cK$.
The {\bf Salvetti poset} $\Sal(\cL)$ is the set of pairs $(X,T)$ where $X\leq T\in \cL$ and $T$ is a tope.
The partial order is given as follows:
$$(X,T)\preceq(X',T') \quad\Leftrightarrow\quad X\leq X'\;\;\text{and}\;\; X'\circ T = T'.$$
We write $\Sal(\cA,\cK) := \Sal(\cL(\cA,\cK))$ for brevity.

\begin{remark}
The definition that we have given is opposite to the usual definition of the Salvetti poset of an oriented matroid.
Our definition is better adapted to Lemma \ref{containment} and to the proof of Theorem \ref{thm:z-pair}.
Note that this has no effect on the definition of the Salvetti complex, as the order complex of a poset is canonically isomorphic to that of its opposite.
\end{remark}

\begin{example}\label{two points 2}
Let $\cA$ be the arrangement from Example \ref{two points 1}.  
The poset $\Sal(\cA,V)$ consists of seven elements, and again the maximal chains have length 2, so the order complex looks like the Hasse diagram.
On the left we draw the diagram with the maximal elements on the top row and colored green, and on the right we redraw it without crossings but with the maximal
elements in the middle.  One can see that it is homotopy equivalent to $\cM(\cA,V)$.
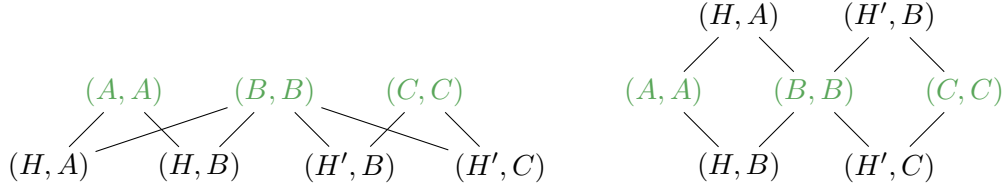
\begin{center}
\begin{tikzpicture}[]
\node[invisivertex ] (HA) at (0,0){$(H,A)$};
\node[invisivertex ] (HB) at (2,0){$(H,B)$};
\node[invisivertex ] (H'B) at (4,0){$(H',B)$};
\node[invisivertex ] (H'C) at (6,0){$(H',C)$};

\node[invisivertex,color=green!30!gray] (AA) at (1,1){$(A,A)$};
\node[invisivertex,color=green!30!gray] (BB) at (3,1){$(B,B)$};
\node[invisivertex,color=green!30!gray] (CC) at (5,1){$(C,C)$};

\path[-] (HA) edge [] node[above] {} (AA);
\path[-] (HB) edge [] node[above] {} (AA);
\path[-] (HA) edge [] node[above] {} (BB);
\path[-] (HB) edge [] node[above] {} (BB);

\path[-] (H'B) edge [] node[above] {} (BB);
\path[-] (H'C) edge [] node[above] {} (BB);
\path[-] (H'B) edge [] node[above] {} (CC);
\path[-] (H'C) edge [] node[above] {} (CC);
\end{tikzpicture}
\qquad
\begin{tikzpicture}[]
\node[invisivertex] (HA) at (1,1){$(H,A)$};
\node[invisivertex] (HB) at (1,-1){$(H,B)$};
\node[invisivertex] (H'B) at (3,1){$(H',B)$};
\node[invisivertex] (H'C) at (3,-1){$(H',C)$};

\node[invisivertex,color=green!30!gray] (AA) at (0,0){$(A,A)$};
\node[invisivertex,color=green!30!gray] (BB) at (2,0){$(B,B)$};
\node[invisivertex,color=green!30!gray] (CC) at (4,0){$(C,C)$};

\path[-] (HA) edge [] node[above] {} (AA);
\path[-] (HB) edge [] node[above] {} (AA);
\path[-] (HA) edge [] node[above] {} (BB);
\path[-] (HB) edge [] node[above] {} (BB);

\path[-] (H'B) edge [] node[above] {} (BB);
\path[-] (H'C) edge [] node[above] {} (BB);
\path[-] (H'B) edge [] node[above] {} (CC);
\path[-] (H'C) edge [] node[above] {} (CC);
\end{tikzpicture}
\end{center}
\end{example}

\section{A nerve lemma}
Let $\Top$ denote the category of topological spaces.
If $\cat{C}$ is a small category and $\Phi:\cat{C}\to\Top$ is a functor, the {\bf simplicial replacement} of $\Phi$ is the
simplicial space $\srep_\bdot(\Phi)$ given by
\[ [n]\mapsto \coprod_{x_0\ra x_1\ra \cdots \ra x_n} \Phi(x_0),
\]
where the coproduct is
indexed over chains of composable maps in $\cat{C}$, and the face
(respectively degeneracy) maps are induced by deleting (respectively repeating)
objects in the chain.  
By the {\bf homotopy
colimit} of $\Phi$, we mean the geometric realization
\[ \hocolim_\cat{C} \Phi :=\bigl|\srep_\bdot(\Phi)\bigr|
\]
of the simplicial replacement.  From certain
perspectives this is only a specific model for the homotopy colimit,
but we will not need any models other than this one.

\begin{example}\label{classifying}
Let $\Psi$ denote the
constant functor on $\cat{C}$ whose value is the 1-point space.  Then $\srep_\bdot(\Psi)$
is the nerve of $\cat{C}$ and $|\cat{C}| := \hocolim_\cat{C} \Psi$ is the classifying space
of $\cat{C}$.  If $P$ is a poset, which we regard as a category with a unique morphism from $x$ to $y$ whenever
$x\leq y$, then the classifying space of $P$ coincides with the order complex of $P$, justifying the fact that we denote both by $|P|$.
\end{example}

If $\Phi,\Phi': \cat{C}\ra \Top$ are two functors and $\alpha: \Phi\ra
\Phi'$ is a natural transformation, there is an induced map
$\alpha_*: \hocolim_\cat{C} \Phi \ra \hocolim_\cat{C} \Phi'$.  For example, if we take $\Phi'=\Psi$ to be the constant 
functor from Example \ref{classifying}, we obtain a map $\hocolim_\cat{C} \Phi\ra
|\cat{C}|$.  Recall that a {\bf weak homotopy equivalence} is a map that induces isomorphisms on $\pi_0$ and all homotopy groups.
The following standard result follows in this level of generality from \cite[Theorem A.7]{DI}.

\begin{theorem}\label{whe-whe}
If $\alpha(x): \Phi(x)\ra \Phi'(x)$ is a weak homotopy equivalence for
every object $x\in \cat{C}$, then $\alpha_*: \hocolim_\cat{C} \Phi \ra \hocolim_\cat{C} \Phi'$ is also a weak homotopy
equivalence.
\end{theorem}

For any space $Z$, let $\Op(Z)$ denote the poset of
open sets of $Z$, ordered by inclusion.
Let $P$ be a finite poset and let $\Phi: P\ra \Op(Z)$ be a map of posets.
The canonical map from the homotopy colimit to the colimit induces a {\bf homotopical collapse map} $f_\Phi:\hocolim_P \Phi\to Z$.
The following proposition appears in \cite[Proposition A.5]{Segal}; see \cite{Dugger} for a more detailed proof.

\begin{proposition}
\label{th:main}
For each point $z\in Z$, let $P_z$ be the sub-poset of $P$ consisting of 
elements $x\in P$ such that $z\in \Phi(x)$.  If $|P_z|$ is contractible for every $z\in Z$, then the homotopical collapse map $f_\Phi:\hocolim_P \Phi\ra Z$ is a weak homotopy equivalence.
\end{proposition}

As a corollary, we obtain the following variant of the nerve lemma.

\begin{corollary}\label{Dan's got a lot of nerve}
Let $P$ and $\Phi$ be as above, with $|P_z|$ contractible
for every $z\in Z$.  If we further assume that $\Phi(x)$ is
contractible for every $x\in P$, then $Z$ has the same weak homotopy class as $|P|$.
\end{corollary}

\begin{proof}
Example \ref{classifying} and Theorem \ref{whe-whe} together imply that the induced map $\hocolim_P\Phi\to |P|$
is a weak homotopy equivalence, and 
Proposition \ref{th:main} asserts that the homotopical collapse map $f_\Phi:\hocolim_P\Phi\to Z$ is also a weak homotopy equivalence.
\end{proof}

\section{The space \boldmath{$\cZ(\cA,\cK)$}}
Let $\cT$ be the set of topes of $\cL(\cA,\cK)$ and define $$\cZ(\cA,\cK) := \cK\times\cT/\sim,$$
where $(v, T) \sim (v',T')$ if and only if $v=v'$ and $v\notin H_e$ for all $e\in\Sep(T,T')$.
That is, we take one copy of the convex open set $\cK$ for each tope $T$, and we glue them together along the complement of the hyperplanes
that separate the associated regions $F_T\cap\cK$.
In the case where $\cK=V$, this definition has appeared in \cite{non-Haus} and \cite{bug}.  There is a natural map $\pi:\cZ(\cA,\cK)\to\cK$ taking
$(v,T)$ to $v$.  For any $v\in F_X\cap\cK$, the fiber $\pi^{-1}(v)$ is a finite set in bijection with the set of topes $T$ with $X\leq T$. 
This is because, for any $T$ (not necessarily lying above $X$), we have $(v,T)\sim (v, X\circ T)$ and $X\leq X\circ T$.

\begin{example}\label{two points 3}
Returning to the example discussed in Examples \ref{two points 1} and \ref{two points 2},
the space $\cZ(\cA,V)$ is a line with two double points:
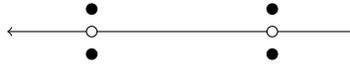
\begin{center}
      \begin{tikzpicture}[scale=.6]
    
       \node[invisivertex] (R-) at (-4,0){};
       \node[invisivertex] (R+) at (4,0){};
    
       \node[point] (H) at (-2,-.5){};
       \node[point] (H) at (-2,.5){};
    
       \node[point] (H') at (2,-.5){};
       \node[point] (H') at (2,.5){};
    
       \path[<->] (R-) edge [] node[above] {} (R+);
       \node[puncture] (H) at (-2,0){};
       \node[puncture] (H') at (2,0){};
    \end{tikzpicture}
    \end{center}
The images in $\cZ(\cA,V)$ of $V\times \{A\}$, $V\times\{B\}$, and $V\times\{C\}$ are as follows:
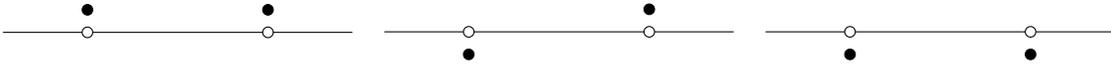
\begin{center}
    \begin{tikzpicture}[scale=.6]
     \node[invisivertex] (R-) at (-4,0){};
     \node[invisivertex] (R+) at (4,0){};
  
     \node[point] (H) at (-2,.5){};
  
     \node[invisivertex] (H') at (2,-.5){};
     \node[point] (H') at (2,.5){};
  
     \path[-] (R-) edge [] node[above] {} (R+);
     \node[puncture] (H) at (-2,0){};
     \node[puncture] (H') at (2,0){};
  \end{tikzpicture}
  \begin{tikzpicture}[scale=.6]
     \node[invisivertex] (R-) at (-4,0){};
     \node[invisivertex] (R+) at (4,0){};
  
     \node[point] (H) at (-2,-.5){};
  
     \node[point] (H') at (2,.5){};
  
     \path[-] (R-) edge [] node[above] {} (R+);
     \node[puncture] (H) at (-2,0){};
     \node[puncture] (H') at (2,0){};
  \end{tikzpicture}
  \begin{tikzpicture}[scale=.6]
     \node[invisivertex] (R-) at (-4,0){};
     \node[invisivertex] (R+) at (4,0){};
  
     \node[point] (H) at (-2,-.5){};
  
     \node[point] (H') at (2,-.5){};
  
     \path[-] (R-) edge [] node[above] {} (R+);
     \node[puncture] (H) at (-2,0){};
     \node[puncture] (H') at (2,0){};
  \end{tikzpicture}
  \end{center}
\end{example}

Recall that we have defined $$\cM(\cA,\cK) := \{v+iw\in M(\cA)\mid v\in\cK\},$$ and consider the projection $\rho:\cM(\cA,\cK)\to\cK$ given
by putting $\rho(v+iw)=v$.  We formulate and prove the following slightly stronger version of Theorem \ref{thm:z-pair}.

\begin{theorem}\label{stronger}
There exists a canonical map $\varphi:\cM(\cA,\cK)\to\cZ(\cA,\cK)$ that is a principal bundle with structure group $V^*$, and furthermore has the property
that $\pi\circ\varphi=\rho$.
\end{theorem}

\begin{proof}
In the case where $\cK=V$, a map $\varphi:\cM(\cA,V)\to\cZ(\cA,V)$ is constructed in \cite{non-Haus}.  
Now we are considering an open set $\cZ(\cA,\cK) = \rho^{-1}(\cK)\subset\cZ(\cA,V)$,
along with its preimage $$\cM(\cA,\cK) = \pi^{-1}(\cK) = \varphi^{-1}\rho^{-1}(\cK) = \varphi^{-1}(\cZ(\cA,\cK))\subset\cM(\cA,V).$$
The restriction of a principal bundle to an open subset is again a principal bundle.
\end{proof}

\vspace{-\baselineskip}

\begin{example}\label{two points 4}
Consider again our running example of two points in a line.  The map from $\cM(\cA,V)$ to $\cZ(\cA,V)$ can be visualized as squishing $\cM(\cA,V)$ vertically:
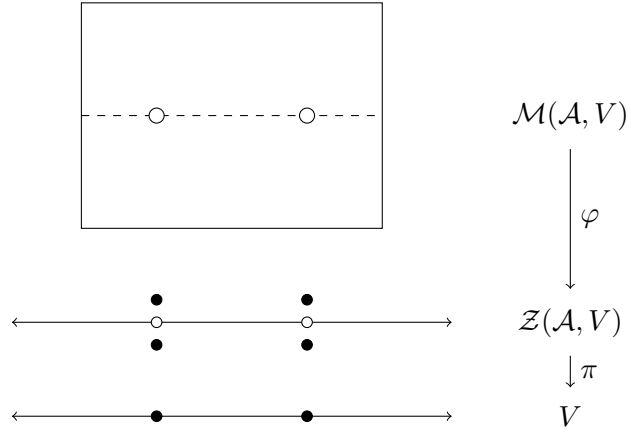
\begin{center}
\begin{tikzpicture}[]
\node[invisivertex,inner sep=7] (labelM) at (4.5,-1){$\cM(\cA,V)$};
\draw[] (-2,-2.5) rectangle (2,.5);
\path[dashed] (-2,-1) edge (2,-1);
\draw[fill=white] (-1,-1) circle (0.1);
\draw[fill=white] (1,-1) circle (0.1);

\node[invisivertex,inner sep=7] (labelZ) at (4.5,-3.75){$\cZ(\cA,V)$};
\node[invisivertex] (R-) at (-3,-3.75){};
\node[invisivertex] (R+) at (3,-3.75){};
\node[point] (H) at (-1,-3.45){};
\node[point] (H) at (-1,-4.05){};
\node[point] (H') at (1,-3.45){};
\node[point] (H') at (1,-4.05){};
\path[<->] (R-) edge [] node[above] {} (R+);
\node[puncture] (H) at (-1,-3.75){};
\node[puncture] (H') at (1,-3.75){};

\node[invisivertex,inner sep=7] (labelV) at (4.5,-5){$V$};
\node[invisivertex] (2R-) at (-3,-5){};
\node[invisivertex] (2R+) at (3,-5){};
\path[<->] (2R-) edge [] node[above] {} (2R+);
\node[point] (2H) at (-1,-5){};
\node[point] (2H') at (1,-5){};

\path[->,
] (labelM) edge 
node[right] {$\varphi$} (labelZ);
\path[->,
] (labelZ) edge 
node[right] {$\pi$} (labelV);
\end{tikzpicture}
\end{center}
It is clear that the fiber over every point of $\cZ(\cA,V)$ is homeomorphic to a line.  It is not obvious (but it is true) that this is a fiber
bundle, and in fact a principal bundle whose structure group can be canonically identified with $V^*$.
\end{example}

For any covector $X\in\cL$ and tope $T\in\cT$, let
$$\cK_{X,T} := \cK\cap \bigcap_{X_e=0}H_e^{T_e} \subset\cK$$
be the unique component of $\cK\setminus\bigcup_{X_e=0}H_e$ containing $\cK\cap F_T$.
(Note that this makes sense regardless of whether or not $X\leq T$.)
Let
$\tilde U_{X,T} := \cK_{X,T} \times \{T\} \subset \cK\times\cT$,
and let $U_{X,T}$ be the image of $\tilde U_{X,T}$ in $\cZ(\cA,\cK)$.
Since $U_{X,T}\cong\tilde U_{X,T}\cong\cK_{X,T}$, it is contractible.

\begin{example}\label{two points 5}
Returning to the example of two points in a line, and labeling the
covectors as in Example \ref{two points 1}, we have the following picture of $U_{H,B}$:
\begin{center}
  \begin{tikzpicture}[scale=.6]
     \node[invisivertex] (R-) at (-4,0){};
     \node[invisivertex] (RH) at (-2,0){};
     \node[invisivertex] (R+) at (4,0){};
  
  
     \node[point,color=green!30!gray] (H') at (2,.5){};
  
    \path[-] (R-) edge [] node[above] {} (RH);
    \path[-] (RH) edge [green!30!gray,ultra thick] node[above] {} (R+);
     \node[puncture] (H) at (-2,0){};
     \node[puncture] (H') at (2,0){};
  \end{tikzpicture}
\end{center}
\end{example}

\begin{lemma}\label{containment}
If $(X,T)\preceq(X',T')$, then $U_{X,T}\subset U_{X',T'}$.
\end{lemma}

\begin{proof}
Suppose that $(X,T)\preceq(X',T')$.  This implies that, for all $e$ such that $X'_e=0$, we have $X_e=0$ and $T_e=T'_e$.
The result follows.
\end{proof}
\vspace{-\baselineskip}

\begin{lemma}\label{Kz}
Suppose that $v\in F_X\cap\cK$.  Then $(v,T)\sim(v,T')$ if and only if $\cK_{X,T} = \cK_{X,T'}$.
\end{lemma}

\begin{proof}
Both conditions are equivalent to the statement that the chambers $F_T$ and $F_{T'}$ lie on the same side of every hyperplane containing $v$.
\end{proof}
\vspace{-\baselineskip}

\begin{example}
Continuing our running example, we illustrate
the cone $\cK_{H,A}$ on the left and the cone $\cK_{H,B} = \cK_{H,C}$ on the right:
  \begin{center}
  \begin{tikzpicture}[]
     \node[invisivertex] (R-) at (-2,0){};
     \node[invisivertex] (RH) at (-1,0){};
     \node[invisivertex] (R+) at (2,0){};
  
    \path[<-] (R+) edge [] node[above] {} (RH);
    \path[->] (RH) edge [green!30!gray,ultra thick] node[above] {} (R-);
     \node[puncture] (H) at (-1,0){};
     \node[point] (H') at (1,0){};
  \end{tikzpicture}
  \qquad   \qquad
  \begin{tikzpicture}[]
     \node[invisivertex] (R-) at (-2,0){};
     \node[invisivertex] (RH) at (-1,0){};
     \node[invisivertex] (R+) at (2,0){};
  
    \path[<-] (R-) edge [] node[above] {} (RH);
    \path[->] (RH) edge [green!30!gray,ultra thick] node[above] {} (R+);
     \node[puncture] (H) at (-1,0){};
     \node[point,color=green!30!gray] (H') at (1,0){};
  \end{tikzpicture}
\end{center}
Writing $F_H = \{v\}$, we have $(v,A)\not\sim (v,B) \sim (v,C)$ because $F_A$ lies on one side of $F_H$ while $F_B$ and $F_C$ lie on the other side.
\end{example}

Given an element $z\in\cZ(\cA,\cK)$ represented by a pair $(v,T)\in\cK\times\cT$, let $X\in\cL$ be the unique covector such that $v\in F_X$,
and let $\cK_z := \cK_{X,T}$.  Lemma \ref{Kz} tells us that this is well defined.  

\begin{lemma}\label{confusing}
Let $(X,T)$ and $(Y,S)$ be elements of $\Sal(\cA,\cK)$.  Fix an element $v\in F_Y$, and let $z\in\cZ(\cA,\cK)$ be the element represented
by the pair $(v,S)\in\cK\times\cT$.  The following are equivalent:
\begin{enumerate}
\item[{\em (1)}] $T = X\circ Y$ and $S = Y\circ X$
\item[{\em (2)}] $T = X\circ Y$ and $F_X\cap \cK\subset\cK_z$
\item[{\em (3)}] $F_Y$ and $F_T$ lie on the same side of every hyperplane containing
$F_X$, and $F_X$ and $F_S$ lie on the same side of every hyperplane containing $F_Y$
\item[{\em (4)}] $z\in U_{X,T}$.
\end{enumerate}
\end{lemma}

\begin{proof}
The equivalence of (1), (2), and (3) is immediate from the definitions.  The equivalence of the fourth condition to the other three is slightly trickier.

The statement that $z\in U_{X,T}$ is equivalent to the statement that $v\in \cK_{X,T}$ and $(v,S)\sim(v,T)$.
We have $v\in \cK_{X,T}$ if and only if $F_Y$ and $F_T$ lie on the same side of every hyperplane containing $F_X$,
which in particular implies that there is no hyperplane containing both $F_X$ and $F_Y$.
We have $(v,S)\sim(v,T)$ if and only if $F_S$ and $F_T$ lie on the same side of every hyperplane containing $F_Y$.
Given the statement that there is no hyperplane containing both $F_X$ and $F_Y$, along with the fact that $F_X$ lies on the boundary of $F_T$,
this is equivalent to the condition that $F_S$ and $F_X$ lie on the same side of every hyperplane containing $F_Y$.
This establishes the equivalence of (3) and (4).
\end{proof}

\vspace{-\baselineskip}
\begin{proposition}\label{whe}
The Salvetti complex $|\Sal(\cA,\cK)|$ is weakly homotopy equivalent to $\cZ(\cA,\cK)$.
\end{proposition}

\begin{proof}
Consider the map that associates to any element $(X,T)\in\Sal(\cA,\cK)$
the open subset $U_{X,T}\subset\cZ(\cA,\cK)$.  By Lemma \ref{containment}, this is a map of posets. 
For any $z\in \cZ(\cA,\cK)$, consider the poset $$\Sal(\cA,\cK)_{z} = \{(X,T)\in\Sal(\cA,\cK)\mid z\in U_{X,T}\}.$$
Let $Y$ be the unique covector such that $\pi(z)\in F_Y$.
By the equivalence of (2) and (4) in Lemma \ref{confusing}, we have
$$\Sal(\cA,\cK)_{z} = \{(X,X\circ Y)\mid F_X\cap \cK\subset\cK_z\}.$$
In particular, there is a bijection from $\cL(\cA,\cK_z)$ to $\Sal(\cA,\cK)_{z}$ given by the map $X\mapsto (X,X\circ Y)$.
We claim that this bijection is in fact an isomorphism of posets.  Indeed, if
$X\leq X'\in \cL(\cA,\cK_z)$, then $X'\circ X = X'$, which implies that 
$(X,X\circ Y)\preceq(X',X'\circ Y)\in\Sal(\cA,\cK)_{z}$.

Since we have a poset isomorphism $\cL(\cA,\cK_z) \cong \Sal(\cA,\cK)_{z}$, 
Proposition \ref{face poset} implies that the order complex $|\Sal(\cA,\cK)_{z}|$ is contractible.
Then Corollary \ref{Dan's got a lot of nerve} implies that $|\Sal(\cA,\cK)|$ is weakly homotopy equivalent to $\cZ(\cA,\cK)$.
\end{proof}

\begin{proof-z-pair}
Theorem \ref{stronger} and Proposition \ref{whe} together imply that $\cM(\cA,\cK)$ and $|\Sal(\cA,\cK)|$ have the same weak homotopy type.
Since $\cM(\cA,\cK)$ and $|\Sal(\cA,\cK)|$ are both homotopy equivalent to CW-complexes, they must be homotopy equivalent.
This well-known fact that can be proved by combining Whitehead's theorem \cite[Proposition 4.5]{hatcher} with the existence and
uniqueness of CW-approximations \cite[Proposition 4.13 and Corollary 4.19]{hatcher}.
\end{proof-z-pair}

\bibliographystyle{amsalpha}
\bibliography{bibliography}

\providecommand{\bysame}{\leavevmode\hbox to3em{\hrulefill}\thinspace}
\providecommand{\MR}{\relax\ifhmode\unskip\space\fi MR }
\providecommand{\MRhref}[2]{%
  \href{http://www.ams.org/mathscinet-getitem?mr=#1}{#2}
}
\providecommand{\href}[2]{#2}
\begin{thebibliography}{BCK18}

\bibitem[BCK18]{BCK}
Hans-J\"urgen Bandelt, Victor Chepoi, and Kolja Knauer, \emph{C{OM}s: complexes of oriented matroids}, J. Combin. Theory Ser. A \textbf{156} (2018), 195--237. \MR{3762108}

\bibitem[DI04]{DI}
Daniel Dugger and Daniel~C. Isaksen, \emph{Topological hypercovers and {$\Bbb A^1$}-realizations}, Math. Z. \textbf{246} (2004), no.~4, 667--689. \MR{2045835}

\bibitem[Dug25]{Dugger}
Daniel Dugger, \emph{Segal's open covering theorem}, Unpublished note, available at https://pages.uoregon.edu/ddugger/segal-open.pdf.

\bibitem[Hat02]{hatcher}
Allen Hatcher, \emph{Algebraic topology}, Cambridge University Press, Cambridge, 2002. \MR{1867354}

\bibitem[Par93]{Paris}
Luis Paris, \emph{Universal cover of {S}alvetti's complex and topology of simplicial arrangements of hyperplanes}, Trans. Amer. Math. Soc. \textbf{340} (1993), no.~1, 149--178. \MR{1148044}

\bibitem[Par14]{Paris3}
\bysame, \emph{{$K(\pi,1)$} conjecture for {A}rtin groups}, Ann. Fac. Sci. Toulouse Math. (6) \textbf{23} (2014), no.~2, 361--415. \MR{3205598}

\bibitem[Pro07]{non-Haus}
Nicholas Proudfoot, \emph{A non-{H}ausdorff model for the complement of a complexified hyperplane arrangement}, Proc. Amer. Math. Soc. \textbf{135} (2007), no.~12, 3989--3994. \MR{2341950}

\bibitem[PT21]{bug}
Dan Petersen and Philip Tosteson, \emph{Factorization statistics and bug-eyed configuration spaces}, Geom. Topol. \textbf{25} (2021), no.~7, 3691--3723. \MR{4372639}

\bibitem[Sal87]{Salvetti}
M.~Salvetti, \emph{Topology of the complement of real hyperplanes in {${\bf C}^N$}}, Invent. Math. \textbf{88} (1987), no.~3, 603--618. \MR{884802}

\bibitem[Seg78]{Segal}
Graeme Segal, \emph{Classifying spaces related to foliations}, Topology \textbf{17} (1978), no.~4, 367--382. \MR{516216}

\end{thebibliography}

\end{document}